\documentclass[12pt]{amsart}
\usepackage{amssymb, amsmath}
\title{Algebraic sets with fully characteristic radicals}
\author{M. Shahryari}
\address{M. Shahryari: Department of Pure Mathematics,  Faculty of Mathematical
Sciences, University of Tabriz, Tabriz, Iran}
\email{mshahryari@tabrizu.ac.ir}

\newcommand{\LL}{\mathcal{L}}
\newcommand{\Rad}{\mathrm{Rad}}
\newcommand{\Fn}{F_{\mathbf{V}}(n)}
\newcommand{\VA}{V_A(S)}
\newcommand{\VG}{V_G(S)}

\newtheorem{corollary}{Corollary}

\newtheorem {theorem}{Theorem}

\begin{document}

\maketitle
\begin{abstract}
We obtain a necessary and sufficient condition for an algebraic set in a group to have a fully characteristic radical. As a result, we see that if the radical of a system of equation $S$ over a group $G$ is fully characteristic, then there exists a class $\mathfrak{X}$ of subgroups of $G$ such that elements of $S$ are identities of $\mathfrak{X}$.
\end{abstract}


%1111111111111111111111111111111111111111111111111111111111111111111111111
%%%%%%%%%%%%%%%%%%%%%%%%%%%%%%%%%%%%%%%%%%%%%%%%%%%%%%%%%%%%%%%%%%%%%%
\vspace{3cm}
In this article, our notations are the same as \cite{DMR1}, \cite{DMR2}, \cite{DMR3}, \cite{DMR4} and \cite{ModSH}. The reader should review these references for a complete account of the universal algebraic geometry. We fix an algebraic language $\LL$ and  we  denote equations in the both forms $p\approx q$ and $(p, q)$.  For an algebra $A$ of type $\LL$ and any system of equations $S\subseteq At_{\LL}(x_1, \ldots, x_n)$, we denote the corresponding radical ideal by $\Rad_A(S)$. Recall that
$$
\Rad_A(S)=\{ (p\approx q)\in At_{\LL}(x_1, \ldots, x_n): \VA\subseteq V_A(p\approx q)\},
$$
where $V_A(S)$ and $V_A(p\approx q)$ are the corresponding algebraic sets of $S$ and $p\approx q$, respectively. In other words,  $\Rad_A(S)$ is the set of all atomic logical consequences of $S$ in $A$. In general, one can not give a deductive description of $\Rad_A(S)$, because it depends on the axiomatizablity of the prevariety generated by $A$, \cite{MR}. In this direction, any good description of the radicals is important from the universal algebraic geometric point of view. In this article, we give a necessary and sufficient condition for $\Rad_A(S)$ to be fully characteristic congruence (invariant under all endomorphisms). We apply our main result to obtain connections between radicals, identities, coordinate algebras and relatively free algebras. Although most of the results can be formulate in the general frame of arbitrary algebraic structures, we mainly focus on groups in what follows. As a summary, we give here some results in the case of coefficient free algebraic geometry of groups.

Let $G$ be a group and $E\subseteq G^n$ be an algebraic set (with no coefficients). Then the radical $\Rad(E)$ is a fully characteristic (equivalently verbal) subgroup of the free group $F_n$, if and only if, there exists a family $\{ K_i\}$ of $n$-generator subgroups of $G$ such that $E=\bigcup_iK_i^n$. As a result, we will show that if $\Rad_G(S)$ is a verbal subgroup of $F_n$, then there exists a family $\mathfrak{X}$ of $n$-generator subgroups of $G$ such that $\Rad_G(S)$ is exactly the set of all group identities valid in $\mathfrak{X}$. We also see that under this conditions, there exists a variety $\mathbf{W}$ of groups, such that the $n$-generator relatively free group in $\mathbf{W}$ is the coordinate group of $S$.

\section{Main result}
Suppose $\LL$ is an algebraic language and $\mathbf{V}$ is a variety of algebras of type $\LL$. For a finite set $X=\{ x_1, \ldots, x_n\}$ of variables, $\Fn$ denotes the relatively free algebra of $\mathbf{V}$ generated by $X$. A congruence $R$ in $\Fn$ is called {\em fully characteristic} (or fully invariant), if for all endomorphism $\alpha:\Fn\to \Fn$ and $(p,q)\in R$, we have
$(\alpha(p), \alpha(q))\in R$. The following theorem concerns the situation the radicals of algebraic sets in which are fully characteristic.

\begin{theorem}
Let $A\in \mathbf{V}$ be an algebra and $E\subseteq A^n$ be an algebraic set. Then $\Rad(E)$ is fully characteristic congruence of $\Fn$, if and only if $E=\bigcup_iK_i^n$ for a family $\{ K_i\}$ of $n$-generator subalgebras of $A$.
\end{theorem}

\begin{proof}
Let $\{ K_i\}$ be a family of $n$-generator subalgebras of $A$ and $E=\bigcup_iK_i^n$ be algebraic. Let $\alpha:\Fn\to \Fn$ be an endomorphism and $(p, q)\in \Rad(E)$. Suppose
$$
\alpha(x_i)=v_i(x_1, \ldots, x_n)
$$
and $\overline{a}=(a_1, \ldots, a_n)\in E$. Hence, there is an index $i$ such that $a_1, \ldots, a_n\in K_i$. Note that we have
\begin{eqnarray*}
\alpha(p)(\overline{a})&=&p(v_1, \ldots, v_n)(\overline{a})\\
                       &=&p(v_1(\overline{a}), \ldots, v_n(\overline{a})).
\end{eqnarray*}
Since $K_i$ is a subalgebra, so for any $j$ we have $v_j(\overline{a})\in K_i$, and therefore
$$
(v_1(\overline{a}), \ldots, v_n(\overline{a}))\in K_i^n\subseteq E.
$$
This shows that
\begin{eqnarray*}
p(v_1(\overline{a}), \ldots, v_n(\overline{a}))&=&q(v_1(\overline{a}), \ldots, v_n(\overline{a}))\\
                                               &=&\alpha(q)(\overline{a}),
\end{eqnarray*}
so, we have $(\alpha(p), \alpha(q))\in \Rad(E)$. Conversely, suppose that $\Rad(E)$ is fully characteristic and $v_1, \ldots, v_n\in \Fn$ are arbitrary. Consider the endomorphism $\alpha(x_i)=v_i$. For any $(p, q)\in \Rad(E)$, we have $(\alpha(p), \alpha(q))\in \Rad(E)$, so for any $\overline{a}\in E$, the equality
$$
p(v_1(\overline{a}), \ldots, v_n(\overline{a}))=q(v_1(\overline{a}), \ldots, v_n(\overline{a}))
$$
holds. This shows that
$$
(v_1(\overline{a}), \ldots, v_n(\overline{a}))\in V_A(\Rad(E))=E.
$$
Therefore, for arbitrary $v_1, \ldots, v_n\in \Fn$ and $\overline{a}\in E$, we have
$$
(v_1(\overline{a}), \ldots, v_n(\overline{a}))\in E.
$$
Let $K(\overline{a})$ be the subalgebra generated by $a_1, \ldots, a_n$. We have
$$
(v_1(\overline{a}), \ldots, v_n(\overline{a}))\in K(\overline{a})^n,
$$
and hence
$$
E=\bigcup_{\overline{a}}K(\overline{a})^n.
$$
The proof is now completed.
\end{proof}

As a result, we see that for an arbitrary system of equations $S$, the radical $\Rad_A(S)$ is fully characteristic, if and only if, there exists a family $\{ K_i\}$ of $n$-generator subalgebras of $A$, such that
$$
\VA=\bigcup_iK_i^n.
$$
This shows that $\Rad_A(S)=\bigcap_i\Rad(K_i^n)$. Since for an arbitrary algebra $K$, the radical $\Rad(K^n)$ is the set of identities of $K$ with variables in $X$, so we have the following corollary.

\begin{corollary}
Let $S$ be a system of equations in the variety $\mathbf{V}$ and $A\in \mathbf{V}$. If $\Rad_A(S)$ is a fully characteristic, then there is a family of $n$-generator subalgebras of $A$, say $\mathfrak{X}$, such that
$$
\Rad_A(S)=\mathrm{id}_{\mathfrak{X}}(n).
$$
\end{corollary}

\section{Application to groups}
The first application of our main result is about algebraic sets in nilpotent groups the coefficient-free radicals in which are characteristic in the free group $F_n$. It is well-known that the classes of fully characteristic and verbal subgroups of $F_n$ are the same. But it is a very hard problem to determine the structure of characteristic subgroups in free  groups. In \cite{Lee}, Vaughan-Lee proved that if $C$ is a characteristic subgroup of $F_n$ and $\gamma_n(F_n)\subseteq C$, then $C$ is fully characteristic. We can use this result of Vaughan-Lee together with our main result to prove the following.

\begin{theorem}
Let $G$ be a nilpotent group of class at most $n$ and $A\subseteq G^n$ be an algebraic set. Then $\mathrm{Rad}(A)$ is a characteristic subgroup of $F_n$, if and only if $A=\bigcup_i K_i^n$ for some family $\{ K_i\}$ of $n$-generator subgroups of $G$.
\end{theorem}

\begin{proof}
Let $\mathrm{Rad}(A)$ be a characteristic subgroup. Suppose $w\in \gamma_n(F_n)$. Since $G$ is nilpotent of class at most $n$, so for all $(a_1, \ldots, a_n)\in A$, we have $w(a_1, \ldots, a_n)=1$. This shows that $w\in \mathrm{Rad}(A)$. Therefore, we have $\gamma_n(F_n)\subseteq \mathrm{Rad}(A)$ and hence the radical is a fully characteristic subgroup. Now the assertion follows from our main result.
\end{proof}

We are now going to consider  the case of $G$-groups. Let $G$ be an arbitrary group and $\mathbf{V}$ be the variety of $G$-groups (for the basic notions of $G$-groups, see \cite{BMR1}). Note that in this case we have $\Fn=G[X]$, where $G[X]$ denotes the free product $G\ast F_n$. It is shown in \cite{Ama} that the fully characteristic subgroups of $G[X]$ are exactly the $G$-verbal subgroups. Recall that for an arbitrary set $W$ of group words with coefficients from $G$, variables $T_1, \ldots, T_m$ and  arbitrary $G$-group $H$, the $G$-verbal subgroup corresponding to $W$ is the subgroup $W(H)$ of $H$, generated by the set
$$
\{ w(h_1, \ldots, h_m):\ h_1, \ldots, h_m\in H, w\in W\}.
$$

\begin{corollary}
Let $S\subseteq G[X]$ and $\Rad_G(S)$ be a $G$-verbal subgroup of $G[X]$. Then all elements of $S$ are $G$-identities of $G$.
\end{corollary}

\begin{proof}
Let $E=\VG$. Since $\Rad(E)$ is $G$-verbal, so $E=\bigcup_iK_i^n$, where every $K_i$ is a $G$-subgroup of $G$. But, the only $G$-subgroup of $G$ is $G$ itself. So, we have $E=G^n$ and hence any element of $S$ is a $G$-identity of $G$.
\end{proof}

Suppose $H$ is a  $G$-group. Recall that the coordinate group of a set $E\in H^n$ is defined as
$$
\Gamma(E)=\frac{G[X]}{\Rad(E)}.
$$
Similarly, for a system of equations $S$, we define the coordinate group $\Gamma_H(S)$ to be $\Gamma(V_H(S))$. One of the main problems of the  algebraic geometry over groups  is the investigation of the structure of this coordinate group. We show that if $\Rad_H(S)$ is a $G$-verbal subgroup, then there exists a variety $\mathbf{W}$ of $G$-groups such that $\Gamma_H(S)=F_{\mathbf{W}}(n)$, the $n$-generator relatively free group of $\mathbf{W}$.

\begin{corollary}
Let $H$ be a $G$-group and $S$ be a system of $G$-equations such that $V_H(S)=\bigcup_iK_i^n$ for some family of $n$-generator $G$-subgroups of $H$. Then there is a variety $\mathbf{W}$ such that $\Gamma_H(S)=F_{\mathbf{W}}(n)$. The converse is also true.
\end{corollary}

\begin{proof}
We proved that  $\Rad_H(S)$ is a $G$-verbal subgroup and hence there is a set $W$ of $G$-group words, such that
$$
\Rad_H(S)=W(G[X]).
$$
Let $\mathbf{W}$ be the variety of $G$-groups defined by $W$. Then, we have
\begin{eqnarray*}
\Gamma_H(S)&=&\frac{G[X]}{\Rad_H(S)}\\
           &=&\frac{G[X]}{W(G[X])}\\
           &=&F_{\mathbf{W}}(n).
\end{eqnarray*}
Conversely, suppose there is a variety $\mathbf{W}$ such that $\Gamma_H(S)=F_{\mathbf{W}}(n)$. Let $W$ be the set of all $n$-variable $G$-identities valid in $\mathbf{W}$. Note that $W=W(G[X])$ and
$$
F_{\mathbf{W}}(n)=\frac{G[X]}{W}.
$$
This shows that $\Rad_H(S)=W(G[X])$ is a $G$-verbal subgroup of $G[X]$ and hence $V_H(S)$ has a decomposition $V_H(S)=\bigcup_iK_i^n$ for some family of $n$-generator $G$-subgroups of $H$.
\end{proof}

There are some examples of equational noetherian relatively free groups: ordinary free groups, free elements of the varieties of the form $var(H)$ in which $H$ is equational noetherian and free solvable groups of given fixed derived lengths. As another application, we obtain a sufficient condition for a variety to have equational noetherian relatively free elements.

\begin{corollary}
Let $H$ be a $G$-equational noetherian group and $W$ be a set of $G$-group words such that the $G$-verbal subgroup $W(G[X])$ is the radical of some subset of $H^n$. Let $\mathbf{W}$ be the variety of $G$-groups defined by $W$. Then $F_{\mathbf{W}}(n)$ is $G$-equational noetherian.
\end{corollary}

As a final point, note that one can ask about conditions under which the radical of a system of equations is a characteristic subgroup. It seems that this is more complicated but interesting problem.

\end{document}